\documentclass[12pt]{amsart} 
\usepackage{amsmath,amssymb,amscd,amsthm}
\usepackage{latexsym}
\usepackage{graphicx}
\usepackage[english]{babel}
\usepackage[latin1]{inputenc}       
\usepackage{pgf,pgfarrows,pgfnodes,pgfautomata,pgfheaps}
\usepackage{colortbl}
\usepackage{setspace}

\usepackage{hyperref}

\newtheorem{thmx}{Theorem}

\newtheorem{theorem}{Theorem}[section]
\newtheorem{cor}[theorem]{Corollary}
\newtheorem{proposition}[theorem]{Proposition}
\newtheorem{lemma}[theorem]{Lemma}

\newtheorem{remark}[theorem]{Remark}

\def\irr#1{{\rm Irr}(#1)}
\def\irrr#1#2 {\irr {#1 \mid #2}}
\newcommand{\R}{\mathbb R}

\newcommand{\sfe}{{{\mathbb S}^{n-1}}}

\newcommand{\E}{\mathbb E}

\begin{document}

\title[A universal bound in the dimensional Brunn-Minkowski inequality]{A universal bound in the dimensional Brunn-Minkowski inequality for log-concave measures}
\author[Galyna V. Livshyts]{Galyna V. Livshyts}
\address{School of Mathematics, Georgia Institute of Technology, Atlanta, GA} \email{glivshyts6@math.gatech.edu}

\subjclass[2010]{Primary: 52} 
\date{\today}
\begin{abstract} We show that for any even log-concave probability measure $\mu$ on $\R^n$, any pair of symmetric convex sets $K$ and $L$, and any $\lambda\in [0,1],$
$$\mu((1-\lambda) K+\lambda L)^{c_n}\geq (1-\lambda) \mu(K)^{c_n}+\lambda\mu(L)^{c_n},$$
where $c_n\geq n^{-4-o(1)}.$ This constitutes progress towards the dimensional Brunn-Minkowski conjecture (see Gardner, Zvavitch \cite{GZ}, Colesanti, L, Marsiglietti \cite{CLM}). Moreover, our bound improves for various special classes of log-concave measures.
\end{abstract}
\maketitle

\section{Introduction}

Recall that a measure $\mu$ on $\R^n$ is called log-concave if for all Borel sets $K, L$, and for any $\lambda\in [0,1],$
\begin{equation}\label{log-concave}
\mu(\lambda K+(1-\lambda)L)\geq \mu(K)^{\lambda}\mu(L)^{1-\lambda}.
\end{equation}
Throughout this paper, the measures are usually assumed to be probability measures, i.e. they are such that $\mu(\R^n)=1.$

In accordance with Borell's result \cite{bor}, if a measure $\mu$ has density $e^{-V(x)}$, where $V(x)$ is a convex function on $\R^n$ which is finite on the a set with non-empty interior, then $\mu$ is log-concave. Examples of log-concave measures include Lebesgue volume $|\cdot |$ and the Gaussian measure $\gamma$ with density $(2\pi)^{-n/2} e^{-{x^2}/{2}}$.

A notable partial case of Borell's theorem is the Brunn-Minkowski inequality, proved in the full generality by Lusternik \cite{Lust}, which states:
\begin{equation}\label{BM}
|\lambda K+(1-\lambda)L|\geq |K|^{\lambda}|L|^{1-\lambda},
\end{equation}
which holds for all Borel-measurable sets $K, L$ and any $\lambda\in [0,1]$ (note that Minkowski average of Borel-measurable sets is also necessarily Borel-measurable). Furthermore, due to the $n-$homogeneity of Lebesgue measure, (\ref{BM}) self-improves to an a-priori stronger form
\begin{equation}\label{BM-add}
|\lambda K+(1-\lambda)L|^{\frac{1}{n}}\geq \lambda |K|^{\frac{1}{n}}+(1-\lambda)|L|^{\frac{1}{n}}.
\end{equation}
See an extensive survey by Gardner \cite{Gar} on the subject for more information. 

Gardner and Zvavitch \cite{GZ} and Colesanti, L, Marsiglietti \cite{CFM} conjectured that any even log-concave probability measure $\mu$ enjoys the inequality
\begin{equation}\label{gzconj}
\mu(\lambda K+(1-\lambda)L)^{\frac{1}{n}}\geq \lambda \mu(K)^{\frac{1}{n}}+(1-\lambda)\mu(L)^{\frac{1}{n}},
\end{equation}
for any pair of convex symmetric sets $K$ and $L.$

L, Marsiglietti, Nayar and Zvavitch \cite{LMNZ} showed that this conjecture follows from the celebrated Log-Brunn-Minkowski conjecture of B\"or\"oczky, Lutwak, Yang, Zhang \cite{BLYZ} (see also \cite{BLYZ-1} and \cite{BLYZ-2}, and Milman \cite{Eman1}, \cite{Eman2}); a combination of this result with the results of Saroglou \cite{christos}, \cite{christos1}, confirms (\ref{gzconj}) for unconditional convex bodies and unconditional log-concave probability measures. For rotation-invariant measures, this conjecture was verified locally near any ball by Colesanti, L, Marsiglietti \cite{CLM}. Kolesnikov, L \cite{KolLiv} developed an approach to this question, building up on the past works of Kolesnikov and Milman \cite{KM1}, \cite{KM2}, \cite{KolMil-1}, 
\cite{KolMilsupernew}, as well as \cite{CLM}, and showed that in the case of the Gaussian measure, for convex sets containing the origin, the desired inequality holds with power $1/2n;$ this is curious, because earlier, Nayar and Tkocz \cite{NaTk} showed that only the assumption of the sets containing the origin is not sufficient for the inequality to hold with a power as strong as $1/n.$ Remarkably, Eskenazis and Moschidis \cite{EM} showed that for the Gaussian measure $\gamma$ and symmetric convex sets $K$ and $L,$ the inequality (\ref{gzconj}) does hold.

For a log-concave probability measure $\mu$ and $a\in \R^+,$ let $p(\mu,a)$ be the largest real number such that for all convex sets $K$ and $L$ with $\mu(K)\geq a$ and $\mu(L)\geq a,$ and every $\lambda\in [0,1]$ one has 
$$\mu(\lambda K+(1-\lambda)L)^{p(\mu,a)}\geq \lambda \mu(K)^{p(\mu,a)}+(1-\lambda)\mu(L)^{p(\mu,a)}.$$
Analogously, we define $p_s(\mu,a)$ as the largest number such that for all convex \emph{symmetric} sets $K$ and $L$ with $\mu(K)\geq a$ and $\mu(L)\geq a,$ and every $\lambda\in [0,1]$ one has 
$$\mu(\lambda K+(1-\lambda)L)^{p_s(\mu,a)}\geq \lambda \mu(K)^{p_s(\mu,a)}+(1-\lambda)\mu(L)^{p_s(\mu,a)}.$$

\medskip
\medskip 


First, we obtain a lower estimate for $p_s(\mu,a)$ for all even log-concave probability measures, which constitutes progress towards the dimensional Brunn-Minkowski conjecture:

\begin{thmx}\label{log-concave-general}
For each $n\geq 1$ there exists a number $c_n>0$ such that for any even log-concave probability measure $\mu$ on $\R^n$, for all symmetric convex sets $K$ and $L$, and any $\lambda\in [0,1],$ one has
$$\mu(\lambda K+(1-\lambda)L)^{c_n}\geq \lambda \mu(K)^{c_n}+(1-\lambda)\mu(L)^{c_n},$$
Namely, we get $c_n=n^{-4-o(1)}$, where $o(1)$ is a positive number which tends to zero as $n\rightarrow\infty,$ and is bounded from above by an absolute constant (that is, by a constant independent of the dimension).
\end{thmx}

In two particular cases, we show tighter bounds. 

\begin{theorem}\label{product}
Let $p\in [1,2]$ and let $d\mu(x)=e^{-\frac{\|x\|_p}{p}}dx$, where\\ $\|x\|_p=\sqrt[p]{|x_1|^p+...+|x_n|^p}$, and $C_{n,p}$ is the normalizing constant. Then for all symmetric convex sets $K$ and $L$ and any $\lambda\in [0,1],$ one has
$$\mu(\lambda K+(1-\lambda)L)^{A(n,p)}\geq \lambda \mu(K)^{A(n,p)}+(1-\lambda)\mu(L)^{A(n,p)},$$
where
$$A(n,p)= \frac{c(p)}{n(\log n)^{\frac{2-p}{p}}},$$
and $c(p)>0$ is an absolute constant independent of the dimension.
\end{theorem}

In another particular case of a class of ``exponential rotation-invariant measures'', we obtain:

\begin{theorem}\label{rot-inv}
Let $d\mu(x)=e^{-\frac{|x|^{p}}{p}}dx$ and $p\geq 1$, where $|\cdot|$ stands for the Euclidean norm, and $C_{n,p}$ is the normalizing constant. Then for all symmetric convex sets $K$ and $L$ and any $\lambda\in [0,1],$ one has
$$\mu(\lambda K+(1-\lambda)L)^{C(p)n^{-2}}\geq \lambda \mu(K)^{C(p)n^{-2}}+(1-\lambda)\mu(L)^{C(p)n^{-2}},$$
where $C(p)$ depends only on $p$.

Furthermore, in the case when $p\in [1,2],$ the power in the inequality above could be taken to be $Cn^{-1-o(1)}$ in place of $Cn^{-2},$ where $o(1)$ is a positive number which tends to zero as $n\rightarrow\infty,$ and is bounded from above by an absolute constant, independent of the dimension
\end{theorem}

\begin{remark} About half a year after the present paper was posted on arXiv, Cordero-Erasquin and Rotem \cite{CERot} obtained a remarkable result which, in particular, implies Theorem \ref{rot-inv}, but none of the other theorems of the present paper.
	
\end{remark}

\medskip

In the case of Lebesgue measure $|\cdot|$, the quantity $p_s(|\cdot|,a)$ does not depend on $a,$ and the question of lower bounding $p_s(|\cdot|,a)$ is equivalent to bounding from below $\inf_{a\in\R} p(|\cdot|,a)$. However, without homogeneity, a universal bound for $\inf_a p(\mu,a)$ may not reflect the correct rate, and may not be applicable to study isoperimetric type questions. For example, in the case of the Gaussian measure $\gamma$, the Ehrhard inequality implies that $p(\gamma,a)\rightarrow_{a\rightarrow 1}\infty$, and in particular, $p_s(\gamma,a)\rightarrow_{a\rightarrow 1}\infty$ (see more at \cite{gal2}).

The convergence $p_s(\mu,a)\rightarrow\infty$ cannot be the case for all log-concave measures, because for Lebesgue measure $|\cdot|,$ we have $p(|\cdot|,a)=p_s(|\cdot|,a)=\frac{1}{n}$ for every $a\in \R;$ same goes to the restriction of the Lebesgue measure to a convex set. However, the phenomenon $p_s(\mu,a)\rightarrow\infty$ is interesting, and we shall now discuss another situation when it holds. In the absence of Ehrhard's inequality, for no measure other than the Gaussian, can such a conclusion be readily drawn.

Recall that a measure $\mu$ with density $e^{-V}$ is called \emph{uniformly strictly log-concave} if $\nabla^2 V\geq k_1 \rm{Id}$, for some $k_1>0$. We shall show

\begin{theorem}\label{meas-simple}
Let $\mu$ be a uniformly strictly log-concave probability measure on $\R^n$ with an even density. Then $p_s(\mu,a)\rightarrow_{a\rightarrow 1}\infty$.
\end{theorem}

In Section 2 we discuss preliminaries. In Section 3 we show an upper bound on the Poincar\'e constant of a restriction of an isotropic log-concave probability measure to a symmetric convex subset. In Section 4 we discuss general log-concave probability measures and prove Theorems \ref{log-concave-general}, \ref{product} and \ref{rot-inv}. In Section 5 we prove Theorem \ref{meas-simple}.

\textbf{Acknowledgement.} The author is grateful to Benjamin Jaye for many fruitful conversations. The author is grateful to Alexander Kolesnikov for teaching her a lot of mathematics, and also for pointing out to her that Proposition 6.3 from \cite{gal2} could be extended to Proposition \ref{prop1}. The author is grateful to Alexandros Eskenazis for bringing to her attention Remark 33 from \cite{ENT}, which has led to the formulation of the Remark \ref{remark-refer}. The author is also extremely grateful to Pierre Bizeul (more details in Remark \ref{pierre}.)

The author is supported by the NSF CAREER DMS-1753260. The author worked on this project while being a Research Fellow at the program in Probability, Geometry, and Computation in High Dimensions at the Simons Institute for the Theory of Computing. The paper was completed while the author was in residence at the Hausdorff Institute of Mathematics at the program in The Interplay between High-Dimensional Geometry and Probability.

\section{Preliminaries.}

Recall that the Brascamp-Lieb inequality (see \cite{BrLi}, or (15) in \cite{CoKl} for the full generality) says that for any locally Lipschitz function $f\in L^2(\R^n)$ and any convex function $V:\R^n\rightarrow\R,$ we have
\begin{equation}\label{BrLi}
\int_{\R^n} f^2 d\mu-\left(\int_{\R^n} f d\mu\right)^2\leq\int_{\R^n} \langle (\nabla^2 V)^{-1}\nabla f,\nabla f\rangle d\mu,
\end{equation}
where $d\mu(x)=e^{-V(x)}dx$, and $\mu$ is a probability measure. Note that the integral on the right hand side makes sense in the almost everywhere sense. The function $e^{-V}$ is called log-concave when $V$ is convex. See Brascamp, Lieb \cite{BrLi}, or e.g. Bobkov, Ledoux \cite{BL1-BrLib}.

Recall that a set $K$ is called convex if together with every pair of points it contains the interval connecting them, and recall that the characteristic function of a convex set is log-concave. As a consequence of (\ref{BrLi}), for any convex body $K$, 
\begin{equation}\label{BrLi-convex}
\mu(K)\int_K f^2 d\mu-\left(\int_K f d\mu\right)^2\leq\mu(K)\int_K \langle (\nabla^2 V)^{-1}\nabla f,\nabla f\rangle d\mu.
\end{equation}
In the case of the standard Gaussian measure $\gamma$, this becomes, for any convex set $K,$
\begin{equation}\label{poinc-sect4}
\gamma(K)\int_K f^2 d\gamma-\left(\int_K f d\gamma\right)^2\leq\gamma(K)\int_K |\nabla f|^2 d\gamma.
\end{equation}
Furthermore, Cordero-Erasquin, Fradelizi and Maurey showed \cite{CFM} that for symmetric convex sets and even $f$,
\begin{equation}\label{poinc-sect4-sym}
\gamma(K)\int_K f^2 d\gamma-\left(\int_K f d\gamma\right)^2\leq\frac{1}{2}\gamma(K)\int_K |\nabla f|^2 d\gamma.
\end{equation}

Next, we state the following result which is well-known to experts; for the proof, see e.g. Lemma 2.14 from \cite{gal2}.

\begin{lemma}\label{moments}
Let $\mu$ be any rotation-invariant probability measure with an absolutely continuous density. Then 
\begin{itemize}
\item For any $q>0,$ and any convex body $K$ containing the origin,
$$\int_K |x|^q d\mu(x)\geq \int_{B(K)} |x|^q d\mu(x),$$
\item For any $q<0,$ and any convex body $K$ containing the origin,
$$\int_K |x|^q d\mu(x)\leq \int_{B(K)} |x|^q d\mu(x),$$
\end{itemize}
where $B(K)$ is the Euclidean ball centered at the origin such that $\mu(B(K))=\mu(K).$
\end{lemma}

The next lemma follows from computations e.g. in Livshyts \cite{Liv1}. We outline the proof for the reader's convenience.
\begin{lemma}\label{maybe}
For $p,q>0,$ for any $R>0,$
$$\frac{1}{\mu(R B_2^n)}\int_{RB_2^n} |x|^q e^{-\frac{|x|^p}{p}} dx\leq c(p,q)n^{\frac{q}{p}}.$$
\end{lemma}
\begin{proof} Let us denote
$$J^p_{k}(R)=\int_0^R t^k e^{-\frac{t^p}{p}} dt.$$
Integrating in polar coordinates, we note that
	$$\frac{1}{\mu(R B_2^n)}\int_{RB_2^n} |x|^q e^{-\frac{|x|^p}{p}} dx =\frac{J^p_{n+q-1}(R)}{J^p_{n-1}(R)}.$$
	Denote also
	$$g^p_k(t)=t^k e^{-\frac{t^p}{p}}.$$
	It was shown in \cite{Liv1} via the Laplace method (see e.g. De Brujn \cite{brujn}), that there exists a constant $C(p,q)>0$ such that for every $R\geq C(p,q)n^{\frac{1}{p}}$,
	$$\frac{J^p_{n+q-1}(R)}{g_{n+q}(n^{1/p})} \in [\frac{c_1}{\sqrt{n}},\frac{c_2}{\sqrt{n}}],$$
	for some $0<c_1<c_2,$ possibly depending on $p$ and $q.$ Thus for $R\geq C(p,q)n^{\frac{1}{p}}$,
$$\frac{J^p_{n+q-1}(R)}{J^p_{n-1}(R)}\leq c'(p,q)n^{\frac{q}{p}}.$$	
	Therefore, the conclusion of the Lemma is verified when $R\geq C(p,q)n^{1/p}.$
	
	In the complementary case when $R\leq C(p,q)n^{1/p},$ we estimate
	$$J^p_{n+q-1}(R)=\int_0^R t^{n+q-1} e^{-\frac{t^p}{p}} dt \leq R^q \int_0^R t^{n-1} e^{-\frac{t^p}{p}} dt=R^q J^p_{n-1}(R)\leq C'(p,q)n^{\frac{q}{p}} J^p_{n-1}(R),$$
	and the Lemma follows.
\end{proof}

\section{General bounds for Poincar\'e constants of restrictions.}

In this section we discuss bounds on Poincar\'e constants of restriction of isotropic log-concave probability measures to convex sets. The estimate relies on techniques from the theory of log-concave measures (see Klartag \cite{Kl}, \cite{Kl2}, \cite{Kl3}, V. D. Milman \cite{KM}, E. Milman \cite{LogLapl}, Barthe \cite{BK}). An interested reader may also check the recent significant progress on the KLS conjecture \cite{KLS} by Jambulapati, Lee, Vempala \cite{JLV}, Klartag, Lehec \cite{KlLe}, Chen \cite{ChenKLS}, which improved up on the past work of Lee, Vempala \cite{LV}, both of these works building up on Eldan's stochastic localization scheme \cite{Eldan}.  

Recall that the Poincar\'e constant of the restriction of a measure $\mu$ onto a set $K$ is the smallest number $C_{poin}(K,\mu)>0$ such that for any function $f\in W^{1,2}(K,\gamma)\cap Lip(K)$,

\begin{equation}\label{poinc-const}
	\mu(K)\int_K f^2 d\mu-\left(\int_K f d\mu\right)^2\leq C^2_{poin}(K,\mu)\mu(K)\int_K |\nabla f|^2 d\mu.
\end{equation}

Recall that a probability measure $\mu$ on $\R^n$ is called isotropic if $\int_{\R^n} x d\mu=0$ and $Cov(\mu)=(\int_{\R^n} x_i x_j d\mu)=Id.$ We show

\begin{theorem}\label{KLS} 
\begin{itemize}
\item Let $\mu$ be a log-concave even isotropic probability measure. Then for any symmetric convex set $K,$ $C_{poin}(\mu, K)\leq Cn,$ where $C>0$ is an absolute constant independent of the dimension. 
\item If, additionally, $\mu$ is rotation-invariant, then $C_{poin}(\mu, K)\leq Cn^{0.5}.$ 
\end{itemize}
\end{theorem}

\begin{remark}\label{pierre} In the earlier version of this paper, the bounds we got were $n^{1+o(1)}$ and $n^{0.5+o(1)}$, and relied on the recent progress on the KLS constant. We are very grateful to Pierre Bizeul for pointing out that our proof already gives the better bound, in view of the older result of \cite{KLS}, namely (\ref{kls-trace}).
\end{remark}

In order to prove the estimates, we start by verifying the following lemma, which is believed to be well known to experts.

\begin{lemma}\label{ballsbodies}
Let $\mu$ be an isotropic probability log-concave measure such that $d\mu(x)=e^{-V(x)}dx$. Then for any $R>0,$ and any $\theta\in\sfe,$ one has
$$\int_0^R t^{n+3} e^{-V(t \theta)} dt\leq Cn^4 \int_0^{R} t^{n-1} e^{-V(t \theta)} dt.$$
If, additionally, $\mu$ is rotation-invariant, then 
$$\int_0^R t^{n+3} e^{-V(t \theta)} dt\leq Cn^2 \int_0^{R} t^{n-1} e^{-V(t \theta)} dt.$$
Here $C$ stands for absolute constants, independent of the dimension, which may change line to line.
\end{lemma}
\begin{proof} Let $t^{\theta}_0(k)\in\R$ be such a number that the function $g^{\theta}_{k}(t)=t^{k} e^{-V(t\theta)}$ is maximized at $t^{\theta}_0(k)$. That is, letting $V_{\theta}(t)=V(\theta t),$ and taking the derivative in $t$ (as in (6) in \cite{Liv}), we see that
$$\frac{d}{dt}V_{\theta}(t^{\theta}_0(k))\cdot t^{\theta}_0(k)=k.$$
Let 
$$J_{k}(\theta):=\int_0^{\infty} t^{k} e^{-V_{\theta}(t)}dt.$$
Then, as it was shown by Klartag and Milman \cite{KM}, (see also Lemma 2 in Livshyts \cite{Liv}),
\begin{equation}\label{J}
\frac{J_{k}(\theta)}{t^{\theta}_0(k) g^{\theta}_{k}(t^{\theta}_0(k))}\in [\frac{1}{k+1},\frac{C}{\sqrt{k}}].
\end{equation}
Ball (Theorem 5 in \cite{Ball-bod}) showed that
$$|x|\left(J_{n+1}\left(\frac{x}{|x|}\right)\right)^{-\frac{1}{n+2}}$$
defines a norm on $\R^n.$ A straightforward computation shows that the unit ball of this norm is an isotropic convex body, provided that $\mu$ is isotropic. Kannan, Lovasz and Simonovits \cite{KLS} showed that any isotropic convex body is contained in a ball of radius at most $Cn.$ Thus for any $\theta\in\sfe,$ $\left(J_{n+1}\left(\theta\right)\right)^{-\frac{1}{n+2}}\geq \frac{1}{Cn},$ or in other words,
\begin{equation}\label{J<n}
J_{n+1}(\theta)\leq (Cn)^{n+2}.
\end{equation}
In addition, one may show, for $c_1,c_2>0,$ that
\begin{equation}\label{Jrat}
\frac{J_{n+c_1}(\theta)}{J_{n-c_2}(\theta)}\leq C'n^{c_1+c_2};
\end{equation}
this follows from the results of Klartag, Milman \cite{KM} or Livshyts \cite{Liv}; alternatively, one may get it from the combination of (\ref{J<n}) with the one-dimensional case of Theorem 3.5.11 from Artstein-Avidan, Giannopolous, Milman \cite{AGM}, applied with $d\mu=e^{-V_{\theta}(t)}1_{\{t>0\}}$, $f=t$, $q=n+c_1$ and $p=n-c_2.$

Combining (\ref{J<n}), (\ref{Jrat}), and the lower bound of (\ref{J}), we get
\begin{equation}\label{refforref}
t^{\theta}_0(n-1) g^{\theta}_{n+1}(t^{\theta}_0(n-1))\leq (Cn)^{n+2}.
\end{equation}
Since $\frac{d}{dt}V_{\theta}(t^{\theta}_0(n-1))\cdot t^{\theta}_0(n-1)=n-1,$ and as the derivative of $V(\theta t)$ in $t$ is non-decreasing, we have 
$$
V(t^{\theta}_0(n-1))\leq V(0)+ t^{\theta}_0(n-1) V'(t^{\theta}_0(n-1))=V(0)+n-1,
$$
and thus 
\begin{equation}\label{g}
	g_{n+1}(t^{\theta}_0(n-1))\leq t^{\theta}_0(n-1)^{n+1} e^{-V(0)-n+1}
\end{equation}
Combining (\ref{refforref}) raised to the power $\frac{1}{n+2},$ and (\ref{g}) with the fact that, by isotropicity and log-concavity, $V(0)\geq 0$ (see e.g. Lemma 5.5 in \cite{LoVe}), we see that
\begin{equation}\label{conseq}
t_0^{\theta}(n-1) \leq C'n.
\end{equation}

Using (\ref{conseq}), we see that if $R<5 t^{\theta}_0(n-1),$ we get
$$\int_0^R t^{n+3} e^{-V(t \theta)} dt\leq CR^4 \int_0^{R} t^{n-1} e^{-V(t \theta)} dt\leq Cn^4 \int_0^{R} t^{n-1} e^{-V(t \theta)} dt.$$

Next, it was shown e.g. by Klartag and Milman \cite{KM}, (also the equation (19) in Livshyts \cite{Liv} is a stronger version of the fact below):
\begin{equation}\label{conc}
\int_{5t_0^{\theta}(k)}^{\infty} t^{k} e^{-V_{\theta}(t)} dt\leq e^{-Ck} J_{k}(\theta).
\end{equation}

Therefore, if $R\geq 5 t^{\theta}_0(n-1),$ using (\ref{Jrat}) with $c_1=3, c_2=1,$ and then using (\ref{conc}), we get
$$\int_0^R t^{n+3} e^{-V(t \theta)} dt\leq \int_0^{\infty} t^{n+3} e^{-V(t \theta)} dt\leq $$$$Cn^4 \int_0^{\infty} t^{n-1} e^{-V(t \theta)} dt\leq C'n^4 \int_0^{R} t^{n-1} e^{-V(t \theta)} dt.$$

In summary, both when $R<5 t^{\theta}_0(n-1)$ and $R\geq 5 t^{\theta}_0(n-1),$ we get the desired conclusion of the first part of the Lemma. 

In the case when $\mu$ is rotation-invariant, its Ball's body is the isotropic ball, and therefore $t^{\theta}_0(n-1)=(1+o(1))\sqrt{n}$ for all $\theta\in\sfe.$ Applying this bound throughout in place of (\ref{J}), we get the second assertion.
\end{proof}

\textbf{Proof of the Theorem \ref{KLS}.} By the result from \cite{KLS} (see also Theorem 2 in Lee and Vempala \cite{LV}), 
\begin{equation}\label{kls-trace}
C_{poin}(\mu,K)\leq  C'\sqrt{Tr\,Cov(\mu,K)},
\end{equation}
where $Cov(\mu,K)$ is the covariance matrix of the restriction of $\mu$ on $K.$ In the case when $\mu$ is even and $K$ is symmetric, one has
$$Cov(\mu,K)_{ij}=\frac{1}{\mu(K)}\int_{K} x_i x_j d\mu(x),$$
and thus
$$Tr\, Cov(\mu,K) =\frac{1}{\mu(K)}\int_{K} |x|^2 d\mu.$$
We write, using polar coordinates:
$$\int_K |x|^4 d\mu(x)=\int_{\sfe} \int_0^{\|\theta\|^{-1}_K} t^{n+3} e^{-V(t\theta)} dtd\theta\leq$$$$ Cn^4 \int_{\sfe} \int_0^{\|\theta\|^{-1}_K} t^{n-1} e^{-V(t\theta)} dtd\theta=Cn^4\mu(K),$$
where the estimate comes from Lemma \ref{ballsbodies}. 

Therefore, 
$$C_{poin}(\mu,K)\leq \sqrt{\frac{1}{\mu(K)}\int_{K} |x|^2 d\mu}\leq \left(\frac{1}{\mu(K)}\int_{K} |x|^4 d\mu\right)^{\frac{1}{4}}\leq(Cn^4)^{\frac{1}{4}}=Cn.$$ 
In the case of rotation-invariant measures, we apply the second assertion of Lemma \ref{ballsbodies}, to get the bound $Cn^{0.5}. \square$

\medskip

\begin{remark} Note that Theorem \ref{KLS} is sharp up to an absolute constant. Indeed, one may find an isotropic convex body $L$ of diameter $Cn,$ and the restriction of the Lebesgue measure on $L$ onto the ``thin'' convex body $K$ approximating its diameter has the Poincar\'e constant of order $n.$ Furthermore, in the case of rotation-invariant measures, the restriction of the Lebesgue measure on the isotropic ball onto its diameter has Poincar\'e constant of order $\sqrt{n}.$

\end{remark}

We note that Theorem \ref{KLS} implies the following fact, which might be known to experts:

\begin{cor}\label{poin-verygen}
	Let $K\subset \R^n$ be a symmetric convex set which is not the whole space. Let $\mu$ be any even log-concave probability measure with $C^2$ density supported on the whole space. Then $C_{poin}(K,\mu)<\infty$. Moreover, $C_{poin}(K,\mu)$ is bounded from above by a constant which only depends on $\mu$ and $n,$ but not on $K.$
\end{cor}
\begin{proof} Indeed, let $T$ be the linear operator which pushes $\mu$ forward to its isotropic position $\tilde{\mu}$ (which exists by the assumptions). Then 
$$C_{poin}(K,\mu)\leq\|T^{-1}\|_{op}C_{poin}(TK, \tilde{\mu}),$$
as can be seen from the definition of the Poincar\'e constant together with the change of variables. By our assumptions, $\|T\|_{op}<\infty$, and by Theorem \ref{KLS}, $C_{poin}(TK, \tilde{\mu})\leq Cn,$ thus the Corollary follows.
\end{proof}

\begin{remark} In the derivation of the Corollary above, it is important that the transformation $T$ depends on $\mu$ but not $K:$ indeed, unless $K$ is bounded, there is no guarantee that one can bring the restriction of $\mu$ onto $K$ into an isotropic position. For example, if $K$ is a half-space and $\mu$ is Gaussian, no linear operator can make the restriction of $\mu$ onto $K$ isotropic.
\end{remark}

\begin{remark}\label{poin-verygen1}
	In fact, in the case when $\mu$ is not even, and $K$ is not symmetric, the assertion of Corollary \ref{poin-verygen} still holds: $C_{poin}(K,\mu)<\infty$. Moreover, $C_{poin}(K,\mu)$ is bounded from above by a constant which only depends on $\mu$ and $n,$ and the relative barycenter of $K$ with respect to $\mu.$ Indeed, the key place where we used symmetry is 
	$$\|Cov(\mu, K)\|_{op}\leq \sqrt{\frac{1}{\mu(K)}\int_K |x|^4 d\mu(x)},$$
	and in the non-symmetric case, this would be replaced with 
	$$\|Cov(\mu, K)\|_{op}=\sup_{\theta\in\sfe} \sqrt{\frac{1}{\mu(K)}\int_K |x-b|^2\langle x-b,\theta\rangle^2 d\mu(x)}\leq $$$$C(b)\sqrt{\frac{1}{\mu(K)}\int_K |x|^4 d\mu(x)}+C_1(b),$$
	for some constants $C(b), C_1(b)\geq 0$ which only depend on $b=\frac{1}{\mu(K)}\int_K x d\mu(x)$, which, in turn, is a finite number. 
	
We remark also that Lemma \ref{ballsbodies}, which was formally obtained under the assumption of symmetry, also holds with the assumption of the origin selected as the barycenter of $K$ with respect to $\mu$. We leave the details to the interested reader. 
\end{remark}

\section{Proof of Theorems \ref{log-concave-general}, \ref{product} and \ref{rot-inv}.}

The proof relies on the ``$L_2$ method'' of obtaining convexity inequalities, previously studied by Kolesnikov and Milman \cite{KM1}, \cite{KM2}, \cite{KM}, \cite{KolMil-1}, \cite{KolMilsupernew}, as well as Livshyts \cite{KolLiv}, Hosle \cite{HKL}, and others. 

We consider a log-concave probability measure $\mu$ on $\R^n$ with an even twice-differentiable density $e^{-V}.$ Consider also the associated operator
$$Lu=\Delta u-\langle \nabla u,\nabla V\rangle.$$
Recall our notation $n_x$ for the normal vector at the point $x\in\partial K$. Recall the following result from \cite{KolLiv}:

\begin{proposition}[KL \cite{KolLiv}]\label{key_prop}
Let $\mathcal{F}$ be a class of convex sets closed under Minkowski interpolation. Suppose for every $C^2$-smooth $K\in\mathcal{F}$, and any $f\in C^1(\partial K)$ there exists a $u\in C^2(K)\cap W^{1,2}(K)$ with $\langle \nabla u,n_x\rangle=f(x)$ on $x\in\partial K$, and such that
$$\frac{1}{\mu(K)}\int_K (\|\nabla^2 u\|^2+\langle \nabla^2 V\nabla u,\nabla u\rangle )d\mu\geq p\left(\E Lu\right)^2+Var(Lu),$$
where the expectation and the variance are with respect to the restriction of $\mu$ onto $K,$ and $\|\nabla^2 u\|$ is the Hilbert-Schmidt (Frobenius) norm of the Hessian matrix of $u$. Then for every pair of $K, L\in\mathcal{F}$ and any $\lambda\in [0,1],$ one has
$$\mu(\lambda K+(1-\lambda)L)^{p}\geq \lambda \mu(K)^{p}+(1-\lambda)\mu(L)^{p}.$$
\end{proposition}

Next, the following Proposition will be the key ingredient for all three theorems \ref{log-concave-general}, \ref{product} and \ref{rot-inv}.

\begin{proposition}\label{prop2}
Suppose $\mu$ is an even log-concave probability measure. Let $K$ be a symmetric convex set in $\R^n$ and let $u:K\rightarrow\R$ be an even function in $W^{2,2}(K)\cap C^2(K)$. Then for any convex symmetric $A\subset K,$ one has
$$\int_K \|\nabla^2 u\|^2 d\mu\geq \frac{\mu(A)}{\mu(K)}\cdot\frac{(\frac{1}{\mu(A)}\int_A Lu \,d\mu)^2}{n+\frac{1}{\mu(A)}\int_A (C^2_{poin}(K,\mu)|\nabla V|^2-2\langle \nabla V,x\rangle) d\mu}.$$
\end{proposition}
\begin{proof}
We write $u=v+t\frac{x^2}{2}$, for some $t\in\R,$ and note that
\begin{equation}\label{Hess-change1}
\|\nabla^2 u\|^2=\|\nabla^2 v\|^2+2t\Delta v+t^2n,
\end{equation}
and
\begin{equation}\label{Lu-change1}
Lu=Lv+tL\frac{x^2}{2}=Lv+tn-t\langle x,\nabla V\rangle.
\end{equation}
Consequently,
\begin{equation}\label{lapl-v}
\Delta v=\langle \nabla V,\nabla v\rangle+Lu-tn+t\langle x,\nabla V\rangle.
\end{equation}
Since $u$ is even, we have that $v$ is also even, and thus, by the symmetry of $K$ and the evenness of $\mu$, we have $\int \nabla v=0$. Therefore, using (\ref{Hess-change1}), and applying the Poincar\'e inequality (\ref{poinc-sect4}) to $\nabla v$, we get
\begin{equation}\label{step1}
\int_K\|\nabla^2 u\|^2 d\mu=\int_K \left(\|\nabla^2 v\|^2+2t\Delta v+t^2n\, \right) d\mu\geq \int_K \left(C^{-2}_{poin}(K,\mu) |\nabla v |^2+2t\Delta v+t^2n \right)\,d\mu.
\end{equation}
Plugging in (\ref{lapl-v}) into (\ref{step1}), and completing the square, we get
\begin{equation}\label{step2}
\int_K\|\nabla^2 u\|^2 d\mu\geq \int_K \left(-t^2C^2_{poin}(K,\mu) |\nabla V|^2+2t(Lu-tn+t\langle x,\nabla V\rangle)+t^2n \right)\,d\mu.
\end{equation}
Since $A\subset K,$ and writing $\int=\frac{1}{\mu(A)}\int_A d\mu,$ we have
\begin{equation}\label{step2-neq}
\frac{1}{\mu(K)}\int_K\|\nabla^2 u\|^2 d\mu\geq \frac{\mu(A)}{\mu(K)}\int_K \left( -t^2C^2_{poin}(K,\mu) |\nabla V|^2+2t(Lu-tn+t\langle x,\nabla V\rangle)+t^2n\right)\,d\mu.
\end{equation}

Plugging the optimal
$$t=\frac{-\int Lu}{n+\int C^2_{poin}|\nabla V|^2-2\langle \nabla V,x\rangle},$$
and simplifying the expression, we conclude the proof.
\end{proof}

\begin{remark} Note that Proposition \ref{prop2}, applied with $K=A$ and $V=0,$ becomes 
\begin{equation}\label{leb-rem}
\int_K \|\nabla^2 u\|^2 dx\geq \frac{1}{n|K|} \left(\int_K \Delta u\right)^2.
\end{equation}
This estimate does not require symmetry or convexity of $K,$ and simply follows point-wise $\|\nabla^2 u\|^2\geq \frac{1}{n} (\Delta u)^2,$ just because for any positive-definite matrix $A$ one has $\|A\|^2_{HS}\geq \frac{1}{n}tr(A)^2.$ Kolesnikov and Milman \cite{KolMil-1} used this estimate to deduce the (usual) Brunn-Minkowski inequality for convex sets, by combining (\ref{leb-rem}) with Proposition \ref{key_prop}, and solving the equation $\Delta u=1$ with an arbitrary Neumann boundary condition. 

In summary, Proposition \ref{prop2} gives the optimal bound in the case of Lebesgue measure. It also boils down to the tight bound of the Proposition 6.3 from \cite{gal2} in the case of the standard Gaussian measure.
\end{remark}

\medskip

\subsection{Proof of Theorem \ref{log-concave-general}} We shall need a couple of facts about isotropic log-concave probability measures. Firstly, combining Lemma 5.4 and Corollary 5.3 from Klartag \cite{Kl}, we note

\begin{lemma}[Klartag \cite{Kl}, a combination of Lemma 5.4 and Corollary 5.3]\label{klartag1}
If $\mu$ on $\R^n$ is an isotropic log-concave probability measure with density $e^{-W},$ then for a sufficiently large absolute constant $\alpha>0,$ the set
$$\{x\in\R^n:\,W(x)\leq W(0)+\alpha n\}$$
a) has measure at least $1-e^{-\alpha n/8};$\\
b) contains the euclidean ball of radius $0.1$. 
\end{lemma}

Next, let us recall a nice and useful Lemma 2.4 from Klartag, E. Milman \cite{LogLapl}, which we slightly modify by introducing another parameter $\lambda$, and thus outline the proof. Recall, for a symmetric convex set $K,$ we define the polar set
$$K^o=\{x\in\R^n:\,\forall y\in K,\,\langle x,y\rangle\leq 1\}.$$

\begin{lemma}[Klartag, E. Milman \cite{LogLapl}, a modification of Lemma 2.4]\label{klartag2}
	Let $W$ be an even convex function from $C^1(\R^n)$. For any $r,q>0$ and any $\lambda\in [0,1]$, one has
	$$\nabla W\left((1-\lambda) \{W\leq q+W(0)\}\right)\subset \left(\frac{1-\lambda}{\lambda}q+r\right)\{W\leq r+W(0)\}^o.$$
\end{lemma}
\begin{proof} Pick any $z\in \{W\leq r+W(0)\}$ and $x\in (1-\lambda) \{W\leq q+W(0)\}.$ We write, in view of the fact that $W$ is even and thus $W(x)-W(0)\geq 0:$
$$\langle \nabla W(x),\lambda z\rangle\leq W(x)-W(0)+\langle \nabla W(x),\lambda z\rangle\leq W(x+\lambda z)-W(0)\leq$$ 
$$ (1-\lambda)\left(W\left(\frac{x}{1-\lambda}\right)-W(0)\right)+\lambda (W(z)-W(0)),$$
where in the last passage we used convexity. Dividing both sides by $\lambda$, and using the choice of $x$ and $z,$ we see
$$\langle \nabla W(x),z\rangle\leq \frac{1-\lambda}{\lambda}q+r,$$
which finishes the proof in view of the definition of duality.
\end{proof}

Next, combining Lemmas \ref{klartag1} and \ref{klartag2}, we get

\begin{cor}\label{klartag-comb}
Let $\mu$ on $\R^n$ be an isotropic log-concave even measure with $C^1$ density $e^{-W}.$ There exists a symmetric convex set $A\subset\R^n$ such that\\
a) $\mu(A)\geq 0.9;$\\
b) For any $x\in A$ we have $|\nabla W|\leq C_1 n^2.$\\
Here $C, C_1$ are absolute constants.
\end{cor}
\begin{proof} We let 
$$A=\frac{n-1}{n}\{W(x)\leq W(0)+\alpha n\},$$ 
with $\alpha>0$ chosen to be a sufficiently large constant. Then, since $W$ is ray-decreasing,
$$\mu(A)\geq \left(1-\frac{1}{n}\right)^n \mu\left(W(x)\leq W(0)+\alpha n\right)\geq 0.9,$$
where in the last step we use a) of Lemma \ref{klartag1}. Thus a)  follows. Next, to get b), we apply Lemma \ref{klartag2} with $\lambda=\frac{1}{n}$, $q=r=\alpha n,$ to get that 
$$\nabla W(A)\subset C'n^2 \{W(x)\leq W(0)+\alpha n\}^o\subset C'' n^2 B^n_2,$$
where in the last step we used b) from Lemma \ref{klartag1}, together with the fact that polarity reverses inclusions. This finishes the proof of b).
\end{proof}

\begin{remark} Arguing along the lines of Section 3, one may show that the set $A$ from Corollary \ref{klartag-comb} has the property that for any symmetric convex set $K,$ we have $\mu(K\cap A)\geq 0.5\mu(K).$
\end{remark}

\textbf{Proof of the Theorem \ref{log-concave-general}.} Note that for any linear operator $T,$ and any pair of convex sets $K$ and $L$, one has $T(K+L)=TK+TL.$ Also, we may assume that $K$ is a $C^2-$smooth strictly-convex bounded set, and in particular, there exists a linear operator pushing forward the restriction of $\mu$ onto $K$ into the isotropic position. Therefore, without loss of generality, we may assume that the measure $\mu|_K=\frac{1}{\mu(K)} 1_K(x) e^{-V(x)} dx$ is isotropic. We may also assume without loss of generality that the density of $\mu$ is $C^1-$smooth. It suffices to show that $p^s_{\mu}(K)\geq n^{-4-o(1)}$ in this situation. 

By the recent result of Chen \cite{KLS} (which built up on the work of Lee-Vempala \cite{LV} and Eldan \cite{Eldan}), we have $C_{poin}(\mu,K)\leq n^{o(1)}.$

Using the fact that $\langle \nabla V,x\rangle\geq 0$ for any even convex function $V$, and applying the Proposition \ref{prop2} with the set $A$ from Corollary \ref{klartag-comb}, we get, for any $u\in W^{2,2}(K),$
\begin{equation}\label{getineq}
	\frac{1}{\mu(K)}\int_K \|\nabla^2 u\|^2 d\mu\geq n^{-4-o(1)} \left(
	\frac{1}{\mu(A)}\int_A Lu\,d\mu\right)^2.
\end{equation}

Recall (see e.g. Theorem 2.11 in \cite{gal2}), that for any $f\in C^1(\partial K)$ there exists a $u\in C^2(K)\cap W^{1,2}(K)$ with $\langle \nabla u,n_x\rangle=f(x)$ on $x\in\partial K$, and such that $Lu=C$, with $C=\frac{\int_{\partial K} f d\mu|_{\partial K}}{\mu(K)}$. Note that $Var(Lu)=Var(C)=0$, and also note that, by convexity of $V,$ we have $\langle \nabla^2 V\nabla u,\nabla u\rangle\geq 0$. Therefore, we get from (\ref{getineq}):
$$\frac{1}{\mu(K)}\int \|\nabla^2 u\|^2+\langle \nabla^2 V\nabla u,\nabla u\rangle d\mu\geq p\left(\int Lu\right)^2+Var(Lu)$$
for $p= n^{-4-o(1)}.$ An application of Proposition \ref{key_prop} concludes the proof. $\square$

\medskip

\subsection{Proof of the Theorems \ref{product} and \ref{rot-inv}.}

Before proceeding with the proof of Theorem \ref{product}, we outline the following corollary of a result by Eskenazis, Nayar, Tkocz \cite{ENT}. Recall that a function $f(x)$ is called unimodular if 
$$f(x)=\int_0^{\infty} 1_{K_t}(x)d\nu(t),$$
for some measure $\nu$ on $[0,\infty)$ and some collection of convex symmetric sets $K_t.$ In particular, any even log-concave function is unimodular.

\begin{lemma}\label{unimod}
For any symmetric convex body $K$, any $p\in [1,2]$ and for any $q>0,$ letting the probability measure $d\mu_p(x)=e^{-\frac{\|x\|_p}{p}} dx$, we have
$$\int_K \|x\|^q_q d\mu_p(x)\leq C(p,q) n \mu_p(K),$$
for some constant $C(p,q)$ which depends only on $p$ and $q.$
\end{lemma}
\begin{proof} Firstly, recall that
\begin{equation}\label{makeref111}
\int_{\R^n}\|x\|^q_q d\mu_p(x)\leq C(p,q)n,
\end{equation}
as follows from Fubini's theorem together with the one-dimensional version of Lemma \ref{maybe}.

Eskenazis, Nayar, Tkocz \cite{ENT} (see also Theorem 19 in \cite{BK}) showed that for any pair of unimodular functions $f$ and $g$,
$$\int_{\R^n} e^{-f}e^{-g}d\mu_p(x)\geq \left(\int_{\R^n} e^{-f}d\mu_p(x)\right)\left(\int_{\R^n} e^{-g}d\mu_p(x)\right).$$
As was noticed by Barthe and Klartag (equation (15) in \cite{BK}) via a classical trick, this implies that 
$$\int_{\R^n} fe^{-g}d\mu_p(x)\leq \left(\int_{\R^n} fd\mu_p(x)\right)\left(\int_{\R^n} e^{-g}d\mu_p(x)\right).$$
We plug the even log-concave (thus, in particular, unimodular) functions $f(x)=\|x\|^q_q$ and $g(x)=-\log 1_K(x)$ into the above inequality, use (\ref{makeref111}), and the lemma follows.
\end{proof}

\medskip

\textbf{Proof of Theorem \ref{product}.} Without loss of generality, let $K$ to be a symmetric $C^2$-smooth convex body. Consider the measure $d\mu_p(x)=e^{-\frac{\|x\|_p}{p}}dx$, for $p\in [1,2]$. Note that in this case,
$$|\nabla V(x)|^2=\|x\|_{2(p-1)}^{2(p-1)},$$
and
$$\langle \nabla V,x\rangle=\|x\|^p_p.$$
It follows from Theorem 1 from Barthe and Klartag \cite{BK} that for any convex set $K,$ the Poincar\'e constant of the restriction of $\mu_p$ on $K$ is bounded from above by $C(\log n)^{\frac{2-p}{2p}}.$ Therefore, denoting $\int=\frac{1}{\mu_p(K)}\int_K d\mu_p$ (as before), we get, by Proposition \ref{prop2}:
$$\int \|\nabla^2 u\|^2\geq \frac{(\int Lu)^2}{n+C (\log n)^{\frac{2-p}{p}}\int \|x\|_{2(p-1)}^{2(p-1)}-2\|x\|^p_p}.$$
By Lemma \ref{unimod}, 
$$\frac{1}{\mu_p(K)}\int_K \|x\|_{2(p-1)}^{2(p-1)} d\mu_p\leq C(p)n,$$
and thus 
$$\int \|\nabla^2 u\|^2\geq \frac{(\int Lu)^2}{C(p)n (\log n)^{\frac{2-p}{p}}}.$$
Recall (see e.g. Theorem 2.11 in \cite{gal2}), that for any $f\in C^1(\partial K)$ there exists a $u\in C^2(K)\cap W^{1,2}(K)$ with $\langle \nabla u,n_x\rangle=f(x)$ on $x\in\partial K$, and such that $Lu=C$, with $C=\frac{\int_{\partial K} f d\mu_p|_{\partial K}}{\mu_p(K)}$. With this choice of $Lu,$ as before, we get
$$\frac{1}{\mu(K)}\int_K \|\nabla^2 u\|^2+\langle \nabla^2 V\nabla u,\nabla u\rangle\geq p^s(K,\mu)\left(\int Lu\right)^2+Var(Lu)$$
for $p^s(K,\mu)= \frac{1}{Cn (\log n)^{\frac{2-p}{p}}}.$ An application of Proposition \ref{key_prop} concludes the proof. $\square$

\medskip
\medskip

Lastly, we show \textbf{the proof of Theorem \ref{rot-inv}.} For $d\mu(x)=e^{-\frac{|x|^p}{p}}dx,$ we note that $V=\frac{|x|^p}{p}+C,$ and thus $\nabla V=|x|^{p-2}x$. By Lemma \ref{moments}, for $R>0$ such that $\mu(R B^n_2)=\mu(K),$ we have
$$\frac{1}{\mu(K)}\int_K |x|^{2p-2} d\mu\leq \frac{1}{\mu(RB^n_2)}\int_{RB^n_2} |x|^{2p-2} d\mu\leq C(p)n^{\frac{2p-2}{p}},$$
where in the last passage we used Lemma \ref{maybe}. 

Next, note that a scaling of $n^{\frac{1}{2}-\frac{1}{p}}$ brings $\mu$ to an isotropic position. Therefore, by the second part of the Theorem \ref{KLS}, 
\begin{equation}\label{KLS-justnow}
C_{poin}(\mu, K)\leq Cn^{\frac{1}{2}}\cdot n^{\frac{1}{p}-\frac{1}{2}}=Cn^{\frac{1}{p}}.
\end{equation}

Therefore, in this case, Proposition \ref{prop2} combined with Theorem \ref{KLS} yields 
$$\int \|\nabla^2 u\|^2\geq \frac{(\int Lu)^2}{n+Cn^{\frac{2}{p}}\int |x|^{2p-2}-2|x|^p}\geq C'n^{-2} \left(\int Lu\right)^2.$$ 
The result now follows from the Proposition \ref{key_prop} in the same manner as before. $\square$

\begin{remark}\label{remark-refer} In the case when $p\in [1,2]$, Remark 33 from Eskenazis, Nayar, Tkocz \cite{ENT} indicates that, similarly to the case of the product measures, for any pair of unimodular functions $f$ and $g$,
$$\int_{\R^n} e^{-f}e^{-g}d\mu(x)\geq \left(\int_{\R^n} e^{-f}d\mu(x)\right)\left(\int_{\R^n} e^{-g}d\mu(x)\right),$$
with $d\mu(x)=e^{-\frac{|x|^p}{p}}dx.$ As was noted by Barthe and Klartag \cite{BK}, this implies that for such $\mu,$ for any symmetric convex set, $C_{poin}(\tilde{\mu},K)\leq c\Phi_{KLS}$, where $\tilde{\mu}$ is the ``isotropic dilate'' of $\mu,$ and $\Phi_{KLS}$ is the KLS constant, which was later shown \cite{KLS} to be bounded by $n^{o(1)}.$ In summary, in place of (\ref{KLS-justnow}) (which followed from Theorem \ref{KLS}), we have
$$C_{poin}(\mu, K)\leq n^{\frac{1}{p}-\frac{1}{2}+o(1)}.$$
This shows, with the same argument as above, that when $p\in [1,2]$ and $d\mu(x)=e^{-\frac{|x|^p}{p}}dx$, one has, for all symmetric convex sets $K$ and $L$ and any $\lambda\in [0,1],$
$$\mu(\lambda K+(1-\lambda)L)^{\frac{1}{n+o(1)}}\geq \lambda \mu(K)^{\frac{1}{n+o(1)}}+(1-\lambda)\mu(L)^{\frac{1}{n+o(1)}}.$$
This implies the ``furthermore'' part of Theorem \ref{rot-inv}.
\end{remark}

\section{Proof of Theorem \ref{meas-simple}.}

Throughout the section, fix a symmetric convex body $K$ and an even log-concave probability measure $\mu$ on $\R^n.$ Recall the notation $\int =\frac{1}{\mu(K)}\int_K d\mu.$

\begin{proposition}\label{prop1}
Let $u\in W^{2,2}(K,\mu)\cap C^2(K)$ be an even function. Then
$$\int tr\left(\nabla^2 u (\nabla^2V)^{-1}\nabla^2 u\right)\geq \int |\nabla u|^2+\frac{(\int Lu)^2}{\int LV}.$$
\end{proposition}
\begin{proof} We write $u=v+tV,$ for some $t\in\R.$ Then
\begin{equation}\label{tr-change}
tr\left(\nabla^2 u (\nabla^2V)^{-1}\nabla^2 u\right)=tr\left(\nabla^2 v(\nabla^2V)^{-1}\nabla^2 v\right)+2t\Delta v+t^2 \Delta V,
\end{equation}
and $Lu=Lv+tLV.$ Consequently,
\begin{equation}\label{lapl-v-1}
\Delta v=\langle \nabla V,\nabla v\rangle+Lu-tLV.
\end{equation}
Since $u$ is even, we have that $v$ is also even, and thus, by the symmetry of $K$ and the evenness of $\mu$, we have $\int \nabla v=0$. Therefore, by (\ref{tr-change}), (\ref{lapl-v-1}) and the Brascamp-Lieb inequality (\ref{BrLi}) applied coordinate-wise to $\nabla v,$ we get 
$$\int tr\left(\nabla^2 u (\nabla^2V)^{-1}\nabla^2 u\right)\geq$$
\begin{equation}\label{step1-1}
 \int |\nabla v|^2+2t(\langle\nabla V,\nabla v\rangle+Lu-tLV)+t^2 \Delta V=\int |\nabla u|^2+2t Lu-t^2LV.
\end{equation}
Plugging the optimal 
$$t=\frac{\int Lu}{\int LV},$$
we get the estimate.

\end{proof}

\medskip

\textbf{Proof of the Theorem \ref{meas-simple}.} Suppose $\nabla^2 V\geq k_1\rm{Id}$ on $K$. Then 
$$tr\left(\nabla^2 u (\nabla^2V)^{-1}\nabla^2 u\right)\leq \frac{\|\nabla^2 u\|^2}{k_1},$$
and therefore, by Proposition \ref{prop1},
\begin{equation}\label{step}
\int \|\nabla^2 u\|^2+\langle \nabla^2 V\nabla u,\nabla u\rangle\geq \int\langle \nabla^2 V\nabla u,\nabla u\rangle+\frac{k_1(\int Lu)^2}{\int LV }.
\end{equation}
As $V$ is convex, we have $\langle \nabla^2 V\nabla u,\nabla u\rangle\geq 0,$ and therefore, (\ref{step}) together with Proposition \ref{key_prop} implies
$$p_s(\mu,a)\geq \frac{k_1}{\inf_{K:\,\mu(K)\geq a} \left(\frac{1}{\mu(K)}\int_K LV d\mu\right)}.$$
Recall that $\int_{\R^n} LV d\mu=0$ (as is verified via the integration by parts), and therefore, by dominated convergence theorem, we get $p(\mu,a)\rightarrow_{a\rightarrow 1} \infty.$ 
$\square$

\end{document}